\documentclass[11pt]{article}
\usepackage{amssymb,amsmath,amsthm}
\usepackage[nobysame]{amsrefs}
\usepackage{tikz}

\newtheorem{theorem}{Theorem}[section]
\newtheorem{lemma}[theorem]{Lemma}

\newtheorem{corollary}[theorem]{Corollary}

\DeclareMathOperator{\diadem}{diadem}
\DeclareMathOperator{\core}{core}
\DeclareMathOperator{\corona}{corona}
\DeclareMathOperator{\nucleus}{nucleus}
\newcommand{\ke}{K\"{o}nig-Egerv\'{a}ry }

\title{On some conjectures concerning critical independent sets of a graph}
\author{Taylor Short\thanks{Supported in part by the NSF DMS under contract 1300547.}\\Department of Mathematics\\University of South Carolina\\shorttm2@mailbox.sc.edu}
\date{}

\begin{document}
\maketitle

\begin{abstract}
Let $G$ be a simple graph with vertex set $V(G)$. A set $S\subseteq V(G)$ is independent if no two vertices from $S$ are adjacent. For $X\subseteq V(G)$, the difference of $X$ is $d(X) = |X|-|N(X)|$ and an independent set $A$ is critical if $d(A) = \max \{d(X): X\subseteq V(G) \text{ is an independent set}\}$ (possibly $A=\emptyset$). Let $\nucleus (G)$ and $\diadem (G)$ be the intersection and union, respectively, of all maximum size critical independent sets in $G$. In this paper, we will give two new characterizations of \ke graphs involving $\nucleus (G)$ and $\diadem (G)$. We also prove a related lower bound for the independence number of a graph. This work answers several conjectures posed by Jarden, Levit, and Mandrescu.

\end{abstract}

\section{Introduction}

In this paper $G$ is a simple graph with vertex set $V(G)$, $|V(G)|=n$, and edge set $E(G)$. The set of neighbors of a vertex $v$ is $N_G(v)$ or simply $N(v)$ if there is no possibility of ambiguity. If $X\subseteq V(G)$, then the set of neighbors of $X$ is $N(X) = \cup _{u\in X} N(u)$, $G[X]$ is the subgraph induced by $X$, and $X^c$ is the complement of the subset $X$. For sets $A,B\subseteq V(G)$, we use $A\setminus B$ to denote the vertices belonging to $A$ but not $B$. For such disjoint $A$ and $B$ we let $(A,B)$ denote the set of edges such that each edge is incident to both a vertex in $A$ and a vertex in $B$.

A \emph{matching} $M$ is a set of pairwise non-incident edges of $G$. A matching of maximum cardinality is a \emph{maximum matching} and $\mu (G)$ is the cardinality of such a maximum matching. For a set $A\subseteq V(G)$ and matching $M$, we say $A$ is \emph{saturated} by $M$ if every vertex of $A$ is incident to an edge in $M$. For two disjoint sets $A,B\subseteq V(G)$, we say there is a matching $M$ of $A$ into $B$ if $M$ is a matching of $G$ such that every edge of $M$ belongs to $(A,B)$ and each vertex of $A$ is saturated. An $M$\emph{-alternating path} is a path that alternates between edges in $M$ and those not in $M$. An $M$\emph{-augmenting path} is an $M$-alternating path which begins and ends with an edge not in $M$.

A set $S\subseteq V(G)$ is \emph{independent} if no two vertices from $S$ are adjacent. An independent set of maximum cardinality is a \emph{maximum independent set} and $\alpha (G)$ is the cardinality of such a maximum independent set. For a graph $G$, let $\Omega (G)$ denote the family of all its maximum independent sets, let 
$$
\core (G) = \bigcap \{S: S\in \Omega (G)\}, \hspace{.5cm} \text{ and} \hspace{.5cm} \corona (G) = \bigcup \{S: S \in \Omega (G)\}.
$$
\noindent See \cite{levit4, boros, short} for background and properties of $\core (G)$ and $\corona (G)$. 

For a graph $G$ and a set $X\subseteq V(G)$, the \emph{difference} of $X$ is $d(X) = |X|-|N(X)|$ and the \emph{critical difference} $d(G)$ is $\max \{d(X): X\subseteq V(G)\}$. Zhang \cite{zhang} showed that $\max \{d(X): X\subseteq V(G)\} = \max \{d(S): S\subseteq V(G) \text{ is an independent set}\}$. The set $X$ is a \emph{critical set} if $d(X) = d(G)$. The set $S\subseteq V(G)$ a \emph{critical independent set} if $S$ is both a critical set and independent. A critical independent set of maximum cardinality is called a \emph{maximum critical independent set}. Note that for some graphs the empty set is the only critical independent set, for example odd cycles or complete graphs. See \cite{zhang,butenko,larson2,larson} for more background and properties of critical independent sets.

Finding a maximum independent set is a well-known {\bf NP}-hard problem. Zhang \cite{zhang} first showed that a critical independent set can be found in polynomial time. Butenko and Trukhanov \cite{butenko} showed that every critical independent set is contained in a maximum independent set, thereby directly connecting the problem of finding a critical independent set to that of finding a maximum independent set. 

For a graph $G$ the inequality $\alpha (G) + \mu (G) \le n$ always holds. A graph $G$ is a \emph{\ke graph} if $\alpha (G) + \mu (G) = n$. All bipartite graphs are \ke but there are non-bipartite graphs which are \ke as well, see Figure \ref{ex1} for an example. We adopt the convention that the empty graph $K_0$, without vertices, is a \ke graph. In \cite{larson} it was shown that \ke graphs are closely related to critical independent sets.

\begin{theorem}\cite{larson} \label{kemaxcrit}
A graph $G$ is \ke if, and only if, every maximum independent set in $G$ is critical.
\end{theorem}

\begin{theorem}\cite{larson}\label{indepdecomp}
For any graph $G$, there is a unique set $X \subseteq V(G)$ such that all of the following hold:

$(i)$ $\alpha (G) = \alpha (G[X]) + \alpha (G[X^c])$,\\
\indent $(ii)$ $G[X]$ is a \ke graph,\\
\indent $(iii)$ for every non-empty independent set $S$ in $G[X^c]$, $|N(S)|\ge |S|,$ and\\
\indent $(iv)$ for every maximum critical indendent set $I$ of $G$, $X=I\cup N(I)$.
\end{theorem}

\noindent Larson in \cite{larson2} showed that a maximum critical independent set can be found in polynomial time. So the decomposition in Theorem \ref{indepdecomp} of a graph $G$ into $X$ and $X^c$ is also computable in polynomial time. Figure \ref{ex2} gives an example of this decomposition, where both the sets $X$ and $X^c$ are non-empty. Recall, for some graphs the empty set is the only critical independent set, so for such graphs the set $X$ would be empty. If a graph $G$ is a \ke graph, then the set $X^c$ would be empty. We adopt the convention that if $K_0$ is empty graph, then $\alpha (K_0) = 0$. 

\begin{figure}[h]
\begin{center}
\begin{tikzpicture}[scale=1]
\draw (0,0) circle (2pt) [fill=black];
\draw (0,1) circle (2pt) [fill=black];
\draw (0,2) circle (2pt) [fill=black];
\draw (1,.5) circle (2pt) [fill=black];
\draw (1,1.5) circle (2pt) [fill=black];
\draw (2,.5) circle (2pt) [fill=black];
\draw (2,1.5) circle (2pt) [fill=black];
\draw (3,.5) circle (2pt) [fill=black];
\draw (3,1.5) circle (2pt) [fill=black];
\draw (4,1) circle (2pt) [fill=black];

\draw (0,0)--(1,.5);
\draw (0,0)--(1,1.5);
\draw (0,1)--(1,.5);
\draw (0,1)--(1,1.5);
\draw (0,2)--(1,.5);
\draw (0,2)--(1,1.5);
\draw (1,1.5)--(2,1.5);
\draw (1,.5)--(2,.5);
\draw (2,.5)--(2,1.5);
\draw (2,.5)--(3,1.5);
\draw (2,.5)--(3,.5);
\draw (3,1.5)--(3,.5);
\draw (2,1.5)--(3,.5);
\draw (4,1)--(3,.5);
\draw (2,1.5)--(3,1.5);
\draw (4,1)--(3,1.5);

\draw (2,-.5) node {$G$};
\draw (-.3,2) node {$a$};
\draw (-.3,1) node {$b$};
\draw (-.3,0) node {$c$};
\draw (1.2,1.8) node {$d$};
\draw (1.2,.3) node {$e$};
\draw (2,1.8) node {$f$};
\draw (3,1.8) node {$g$};
\draw (4.3,1) node {$h$};
\draw (3,.2) node {$i$};
\draw (2,.2) node {$j$};

\end{tikzpicture}
\end{center}
\caption{$G$ has maximum critical independent set $I = \{a,b,c\}$. Theorem \ref{indepdecomp} gives that $X = \{a,b,c,d,e\}$ and $X^c=\{f,g,h,i,j\}$.}
\label{ex2}
\end{figure}
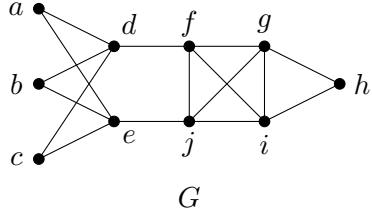

In \cite{levit7,levit2} the following concepts were introduced: for a graph $G$,
\begin{align*}
\ker (G) &= \bigcap \{S: S \text{ is a critical independent set in } G\},\\
\diadem (G) &= \bigcup \{S: S \text{ is a critical independent set in } G\}, \text{ and}\\
\nucleus (G) &= \bigcap \{S: S \text{ is a maximum critical independent set in } G\}.
\end{align*}

\noindent However, the following result due to Larson allows us to use a more suitable definition for $\diadem (G)$.

\begin{theorem}\cite{larson2} \label{larson}
Each critical independent set is contained in some maximum critical independent set.
\end{theorem}

\noindent For the remainder of this paper we define 
$$
\diadem (G) = \bigcup \{S: S \text{ is a maximum critical independent set in } G\}.
$$
Note that if $G$ is a graph where the empty set is the only critical indepedent set (including the case $G = K_0$, the empty graph), then $\ker (G), \diadem (G)$, and $\nucleus (G)$ are all empty. See Figure \ref{ex1} for examples of the sets $\ker (G)$, $\diadem (G)$, and $\nucleus (G)$.

\begin{figure}[h]
\begin{center}
\begin{tikzpicture}[scale=1]
\draw (0,0) circle (2pt) [fill=black];
\draw (0,1) circle (2pt) [fill=black];
\draw (0,2) circle (2pt) [fill=black];
\draw (0,3) circle (2pt) [fill=black];
\draw (1,.5) circle (2pt) [fill=black];
\draw (1,1.5) circle (2pt) [fill=black];
\draw (1,2.5) circle (2pt) [fill=black];

\draw (0,3)--(1,2.5);
\draw (0,2)--(1,2.5);
\draw (0,1)--(1,2.5);
\draw (1,1.5)--(1,.5);
\draw (0,1)--(1,2.5);
\draw (0,1)--(1,1.5);
\draw (0,1)--(1,.5);
\draw (0,0)--(1,.5);

\draw (0.5,-1) node {$G_1$};
\draw (-.3,3) node {$a$};
\draw (-.3,2) node {$b$};
\draw (-.3,1) node {$c$};
\draw (-.3,0) node {$d$};
\draw (1.3,2.5) node {$e$};
\draw (1.3,1.5) node {$f$};
\draw (1.3,.5) node {$g$};

\draw (4,0) circle (2pt) [fill=black];
\draw (4,1) circle (2pt) [fill=black];
\draw (4,2) circle (2pt) [fill=black];
\draw (4,3) circle (2pt) [fill=black];
\draw (5,0.5) circle (2pt) [fill=black];
\draw (5,1.5) circle (2pt) [fill=black];
\draw (5,2.5) circle (2pt) [fill=black];
\draw (6,1) circle (2pt) [fill=black];
\draw (6,2) circle (2pt) [fill=black];
\draw (7,1.5) circle (2pt) [fill=black];

\draw (4,2)--(5,2.5);
\draw (4,1)--(5,2.5);
\draw (4,3)--(5,2.5);
\draw (5,1.5)--(5,.5);
\draw (4,1)--(5,1.5);
\draw (4,0)--(5,.5);
\draw (5,2.5)--(6,2);
\draw (5,.5)--(6,1);
\draw (6,2)--(6,1);
\draw (6,2)--(7,1.5);
\draw (6,1)--(7,1.5);

\draw (5,-1) node {$G_2$};
\draw (3.7,3) node {$a$};
\draw (3.7,2) node {$b$};
\draw (3.7,1) node {$c$};
\draw (3.7,0) node {$d$};
\draw (5.3,2.7) node {$e$};
\draw (5.3,1.5) node {$f$};
\draw (5.3,.3) node {$g$};
\draw (6.2,2.3) node {$h$};
\draw (6.2,.8) node {$i$};
\draw (7.3,1.5) node {$j$};

\end{tikzpicture}
\end{center}
\caption{$G_1$ is a \ke graph with $\ker(G_1) = \{a,b\} \subsetneq \core (G_1) = \nucleus (G_1) = \{a,b,d\}$ and $\diadem (G_1) = \corona (G_1) = \{a,b,c,d,f\}$. $G_2$ is not a \ke graph and has $\ker (G_2) = \core (G_2) = \{a,b\} \subsetneq \nucleus (G_2) = \{a,b,d\}$ and $\diadem (G_2) = \{a,b,c,d,f\} \subsetneq \corona (G) = \{a,b,c,d,f,g,h,i,j\}$.}
\label{ex1}
\end{figure}
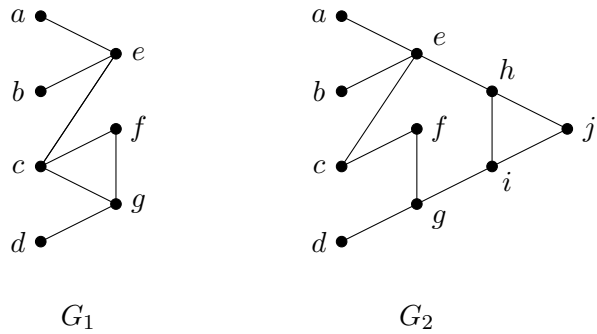

In \cite{levit1,levit2}, the following necessary conditions for \ke graphs were given:

\begin{theorem}\cite{levit1}\label{kediadem}
If $G$ is a \ke graph, then

$(i)$ $\diadem (G) = \corona (G)$, and

\indent $(ii)$ $|\ker (G)| + |\diadem (G)| \le 2\alpha (G)$.
\end{theorem}

\begin{theorem}\cite{levit2} \label{kenucdiadem}
If $G$ is a \ke graph, then $|\nucleus (G)| + | \diadem (G)| = 2\alpha (G)$.
\end{theorem}

\noindent In \cite{levit1} it was conjectured that condition $(i)$ of Theorem \ref{kediadem} is sufficient for \ke graphs and in \cite{levit2} it was conjectured the necessary condition in Theorem \ref{kenucdiadem} is also sufficient. The purpose of this paper is to affirm these conjectures by proving the following new characterizations of \ke graphs.

\begin{theorem} \label{keconjecture}
For a graph $G$, the following are equivalent:

$(i)$ $G$ is a \ke graph, \\
\indent $(ii)$ $\diadem (G) = \corona (G)$, and\\
\indent $(iii)$ $|\diadem (G)| + |\nucleus (G)| = 2\alpha (G)$.
\end{theorem}

The paper \cite{levit1} gives an upper bound for $\alpha (G)$ in terms of unions and intersections of maximum independent sets, proving 
$$
2\alpha (G) \le |\core (G)| + |\corona (G)|
$$ 
for any graph $G$. It is natural to ask whether a similar lower bound for $\alpha (G)$ can be formulated in terms of unions and intersections of critical independent sets. Jarden, Levit, and Mandrescu in \cite{levit1} conjectured that for any graph $G$, the inequality $|\ker (G)| + |\diadem (G)| \le 2\alpha (G)$ always holds. We will prove a slightly stronger statement. By Theorem \ref{larson} we see that $\ker (G) \subseteq \nucleus (G)$ holds implying that $|\ker (G)| + |\diadem (G)| \le |\nucleus (G)| + |\diadem (G)|$. In section \ref{indsection} we will prove the following statement, resolving the cited conjecture:

\begin{theorem} \label{conjecture1}
For any graph $G$,
$$
|\nucleus (G)| + |\diadem (G)| \le 2\alpha (G).
$$
\end{theorem}

\noindent It would be interesting to know whether the sets $\nucleus (G)$ and $\diadem (G)$, or their sizes, can be computed in polynomial time. 

\section{Some structural lemmas}

Here we prove several crucial lemmas which will be needed in our proofs. Our results hinge upon the structure of the set $X$ as described in Theorem \ref{indepdecomp}.

\begin{lemma} \label{diadem}
Let $I$ be a maximum critical independent set in $G$ and set $X = I \cup N(I)$. Then $\diadem (G) \cup N(\diadem (G)) = X$.
\end{lemma}

\begin{proof}
By Theorem \ref{indepdecomp} the set $X$ is unique in $G$, that is, for any maximum critical independent set $S$, $X = S\cup N(S)$. Then $\diadem (G) = X$ follows by definition.
\end{proof}

\begin{lemma} \label{diademG}
Let $I$ be a maximum critical independent set in $G$ and set $X = I \cup N(I)$. Then $\diadem (G) \subseteq \diadem (G[X])$ and $\nucleus (G[X]) \subseteq \nucleus (G)$.
\end{lemma}

\begin{proof}
Let $S$ be a maximum critical independent set in $G$. Using Theorem \ref{indepdecomp} we see that $S$ is a maximum independent set in $G[X]$ and also $G[X]$ is a \ke graph. Then Theorem \ref{kemaxcrit} gives that $S$ must also be critical in $G[X]$, which implies that $\diadem (G) \subseteq \diadem (G[X])$.

Now let $v\in \nucleus (G[X])$. Then $v$ belongs to every maximum critical indepedent set in $G[X]$. As remarked above, since every maximum critical independent set in $G$ is also a maximum critical independent set in $G[X]$, then $v$ belongs to every maximum critical independent set in $G$. This shows that $v\in \nucleus (G)$ and $\nucleus (G[X]) \subseteq \nucleus (G)$ follows.
\end{proof}

\begin{lemma} \label{misneigh}
Suppose $I$ is a non-empty maximum critical independent set in $G$, set $X = I\cup N(I)$, let $A = \nucleus (G) \setminus \nucleus (G[X])$, and let $S$ be a maximum independent set in $G[X]$. For $S'\subseteq S\cap N(A)$, if there exists $A'\subseteq A$ such that $N(A')\cap S \subseteq S'$, then $|S'| \ge |A'|$.
\end{lemma}

\begin{proof}
For $S' \subseteq S\cap N(A)$ suppose such an $A'$ exists. For sake of contradiction, suppose that $|S'| < |A'|$. Since $A' \subseteq \nucleus (G)$, then $A'$ is an independent set. Also since $A' \subseteq \nucleus (G) \subseteq \diadem (G)$, by Lemma \ref{diadem} we have $A' \subseteq X$. Furthermore, since $N(A')\cap S \subseteq S'$ then $A'\cup (S \setminus S')$ is an independent set in $G[X]$. Now by assumption $|S'| < |A'|$, so $A'\cup (S \setminus S')$ is an independent set in $G[X]$ larger than $S$, which cannot happen. Therefore we must have $|S'| \ge |A'|$ as desired.
\end{proof}

\begin{lemma} \label{diademX}
Let $I$ be a maximum critical independent set in $G$ and set $X = I\cup N(I)$. Then 
$$
|\nucleus (G)| + |\diadem (G)| \le |\nucleus (G[X]) | + |\diadem (G[X])|.
$$
\end{lemma}

\begin{proof}
First note that if the set $X$ is empty, then by Lemma \ref{diadem} both sides of the inequality are zero. So let us assume that $X$ is non-empty. Now consider the set $A = \nucleus (G) \setminus \nucleus (G[X])$. If this independent set is empty, then $\nucleus (G) = \nucleus (G[X])$ and there is nothing to prove since $\diadem (G) \subseteq \diadem (G[X])$ holds by Lemma \ref{diademG}. If $A$ is non-empty, for each $v\in A$ there is some maximum independent set $S$ of $G[X]$ which doesn't contain $v$. Since $S$ is a maximum independent set there exists $u \in N(v)\cap S$. Since $v\in \nucleus (G)$, then $u$ does not belong to any maximum critical independent set in $G$. Recall by Theorem \ref{indepdecomp} $(ii)$ $G[X]$ is a \ke graph, so Theorem \ref{kemaxcrit} gives that $S$ is a maximum critical independent set in $G[X]$. It follows that $u\in \diadem (G[X]) \setminus \diadem (G)$, which shows each vertex in $A$ is adjacent to at least one vertex in $\diadem (G[X]) \setminus \diadem (G)$.

Now we will show there is a maximum matching from $A$ into $\diadem (G[X])\setminus \diadem (G)$ with size $|A|$. For sake of contradiction, suppose such a matching $M$ has less than $|A|$ edges. Then there exists some vertex $v \in A$ not saturated by $M$. By the above, $v$ is adjacent to some vertex $u \in \diadem (G[X]) \setminus \diadem (G)$. Since $M$ is maximum, $u$ is matched to some vertex $w \in A$ under $M$. Now let $S$ be a maximum independent set of $G[X]$ containing $u$. We now restrict ourselves to the subgraph induced by the edges $(A\cap N(S), S\cap N(A))$, noting this subgraph is bipartite since both $A\cap N(S)$ and $S\cap N(A)$ are independent. In this subgraph, consider the set $\mathcal{P}$ of all $M$-alternating paths starting with the edge $vu$. Note that all such paths must start with the vertices $v,u$, then $w$. Also, such paths must end at either a matched vertex in $A\cap N(S)$ or an unmatched vertex in $S\cap N(A)$. 

We wish to show that there is some alternating path ending at an unmatched vertex in $S\cap N(A)$. For sake of contradiction, suppose all alternating paths end at a matched vertex in $A\cap N(S)$ and let $V(\mathcal{P})$ denote the union of all vertices belonging to such an alternating path. We aim to show this scenario contradicts Lemma \ref{misneigh}. Now clearly we must have $N(V(\mathcal{P})\cap A)\cap S \subseteq V(\mathcal{P})\cap S$, else we could extend an alternating path to any vertex in $(N(V(\mathcal{P})\cap A)\cap S) \setminus (V(\mathcal{P})\cap S)$. Also, since all paths in $\mathcal{P}$ end at a matched vertex in $A\cap N(S)$, then every vertex of $V(\mathcal{P})\cap S$ is matched under $M$, and such a situation should look as in Figure \ref{proofpic}. 

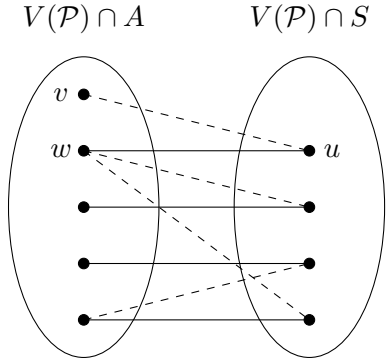
\begin{figure}[h]
\begin{center}
\begin{tikzpicture}[scale=1]

\draw (0,0) circle [x radius=1cm, y radius=2cm];
\draw (3,0) circle [x radius=1cm, y radius=2cm];

\draw (0,1.5) circle (2pt) [fill=black];
\draw (3,.75) circle (2pt) [fill=black];
\draw (0,.75) circle (2pt) [fill=black];
\draw (3,0) circle (2pt) [fill=black];
\draw (3,-.75) circle (2pt) [fill=black];
\draw (3,-1.5) circle (2pt) [fill=black];
\draw (0,0) circle (2pt) [fill=black];
\draw (0,-.75) circle (2pt) [fill=black];
\draw (0,-1.5) circle (2pt) [fill=black];

\draw (-.3,1.5) node {$v$};
\draw (-.3,.75) node {$w$};
\draw (3.3,.75) node {$u$};
\draw (0,2.5) node {$V(\mathcal{P})\cap A$};
\draw (3,2.5) node {$V(\mathcal{P})\cap S$};

\draw [dashed] (0,1.5) -- (3,.75);
\draw [dashed] (0,.75) -- (3,0);
\draw [dashed] (0,.75) -- (3,-1.5);
\draw [dashed] (0,-1.5) -- (3,-.75);
\draw (0,.75) -- (3,.75);
\draw (0,0) -- (3,0);
\draw (0,-.75) -- (3,-.75);
\draw (0,-1.5) -- (3,-1.5);

\end{tikzpicture}
\end{center}
\caption{What the $M$-alternating paths could look like between $V(\mathcal{P})\cap A$ and $V(\mathcal{P})\cap S$, where solid lines represent matched edges in $M$ and dotted lines represent the unmatched edges.}
\label{proofpic}
\end{figure}

\noindent From this it follows that $|V(\mathcal{P})\cap S| < |V(\mathcal{P})\cap A|$. The previous statements exactly contradict Lemma \ref{misneigh}, so there is some alternating path $P$ ending at an unmatched vertex $x \in S\cap N(A)$. This means that $P$ is an $M$-augmenting path. A well-known theorem in graph theory states that a matching is maximum in $G$ if, and only if, there is no augmenting path \cite{west}. So $P$ being an $M$-augmenting path contradicts our assumption that $M$ is a maximum matching.

Therefore there is a matching $M$ from $A$ into $\diadem (G[X]) \setminus \diadem (G)$. This matching implies that $|\nucleus (G) \setminus \nucleus (G[X])| \le |\diadem (G[X]) \setminus \diadem (G)|$. Since both $\nucleus (G[X]) \subseteq \nucleus (G)$ and $\diadem (G) \subseteq \diadem (G[X])$ by Lemma \ref{diademG}, the lemma follows.
\end{proof}

\section{New characterizations of \ke graphs}

\begin{proof}[Proof (of Theorem \ref{keconjecture})]
First we prove $(ii) \Rightarrow (i)$. Suppose that $\diadem (G) = \corona (G)$ holds and let $I$ be a maximum critical independent set with $X = I \cup N(I)$. We will use the decomposition in Theorem \ref{indepdecomp} to show that $X^c$ must be empty and hence, $G=G[X]$ is a \ke graph. By Lemma \ref{diadem} we have $\corona (G) = \diadem (G) \subseteq X$, in other words every maximum independent set in $G$ is contained in $X$. This implies that $|I| = \alpha (G[X]) = \alpha (G)$. Now by Theorem \ref{indepdecomp} $(i)$, $\alpha (G) = \alpha (G[X]) + \alpha (G[X^c])$ showing that we must have $\alpha (G[X^c])=0$. Now clearly the result follows, since $\alpha (G[X^c])=0$ implies that $X^c$ must be empty.

To prove $(iii) \Rightarrow (i)$, again we will use the decomposition in Theorem \ref{indepdecomp} to show that $X^c$ must be empty and hence, $G$ is a \ke graph. So suppose that $|\diadem (G)| + |\nucleus (G)| = 2\alpha (G)$ and let $I$ be a maximum critical independent set in $G$ with $X = I\cup N(I)$. Lemma \ref{diademX} implies that 
$$
2\alpha (G) = |\diadem (G)| + |\nucleus (G)| \le |\diadem (G[X])| + |\nucleus (G[X])|.
$$
Theorem \ref{indepdecomp} $(ii)$ gives that $G[X]$ is \ke, so by Corollary \ref{kenucdiadem} we have $|\diadem (G[X])| + |\nucleus (G[X])|=2\alpha (G[X])$ implying that $\alpha (G) \le \alpha (G[X])$. It follows by Theorem \ref{indepdecomp} $(i)$ we must have $\alpha (G) = \alpha (G[X])$, so again we know that $\alpha (G[X^c])=0$ which finishes this part of the proof.

The implications $(i) \Rightarrow (ii)$ and $(i) \Rightarrow (iii)$ are given in Theorem \ref{kediadem} and in Theorem \ref{kenucdiadem}.
\end{proof}

\section{A bound on $\alpha (G)$}\label{indsection}

\begin{proof}[Proof (of Theorem \ref{conjecture1})]
Let $I$ be a maximum critical independent set in $G$ and $X=I\cup N(I)$. By Theorem \ref{indepdecomp} $(ii)$, $G[X]$ is a \ke graph so by Theorem \ref{kenucdiadem} we have
$$
|\nucleus (G[X])| + |\diadem (G[X])| = 2\alpha(G[X]) \le 2\alpha (G).
$$
Now by Lemma \ref{diademX} we must have 
$$
|\nucleus (G)| + |\diadem (G)| \le |\nucleus (G[X])| + |\diadem (G[X])|
$$
and the theorem follows.
\end{proof}

Combining Theorem \ref{conjecture1} and the inequality $2\alpha (G) \le |\core (G)| + |\corona (G)|$ proven in \cite{levit1}, the following corollary is immediate.

\begin{corollary}
For any graph $G$,
$$
|\nucleus (G)| + |\diadem (G)| \le 2\alpha (G) \le |\core (G)| + |\corona (G)|.
$$
\end{corollary}

\noindent These upper and lower bounds are quite interesting. The fact that every critical independent set is contained in a maximum independent set implies that $\diadem (G) \subseteq \corona (G)$ for all graphs $G$. However, the graph $G_2$ in Figure \ref{ex1} has $\core (G_2) \subsetneq \nucleus (G_2)$ while the graph $G$ in Figure \ref{ex2} has $\nucleus (G) = \{a,b,c\} \subsetneq \core (G) = \{a,b,c,h\}$.

\section{Acknowledgements}
Many thanks to my advisor L\'{a}szl\'{o} Sz\'{e}kely for feedback on initial versions of this manuscript. Partial support from the NSF DMS under contract 1300547 is gratefully acknowledged.

\bibliographystyle{amsplain}
\bibliography{ref}
\nocite{*}

\end{document}